\documentclass[12pt]{article}
\usepackage{amsmath,amsfonts,amssymb,amsthm}
\usepackage{enumerate}

\newtheorem{proposition}{Proposition}[section]
\newtheorem{theorem}[proposition]{Theorem}
\newtheorem{corollary}[proposition]{Corollary}
\newtheorem{lemma}[proposition]{Lemma}
\newtheorem{remark}[proposition]{Remark}
\newtheorem{definition}[proposition]{Definition}
\newtheorem{example}[proposition]{Example}

\newcommand{\nc}{\newcommand}
\nc{\I}{{\mathbf 1}}

\nc{\bN}{{\mathbf N}}
\nc{\bM}{{\mathbf M}}
\nc{\cB}{{\mathcal B}}
\nc{\cX}{{\mathcal X}}
\nc{\cS}{{\mathcal S}}
\nc{\cL}{{\mathcal L}}
\nc{\R}{{\mathbb R}}
\nc{\N}{{\mathbb N}}
\nc{\Z}{{\mathbb Z}}
\nc{\X}{{\mathbb X}}
\nc{\Y}{{\mathbb Y}}
\nc{\sS}{{\mathbb S}}
\nc{\na}{{\rm NA}}
\nc{\pa}{{\rm PA}}
\def\1{\mathbf{1}}
\newcommand{\md}{\text{d}}

\nc{\BP}{\mathbb{P}}
\nc{\BE}{\mathbb{E}}
\nc{\BQ}{\mathbb{Q}}

\DeclareMathOperator{\BC}{{\mathbb Cov}}

\numberwithin{equation}{section}
\newcommand{\remove}[1]{}

\begin{document}

\renewcommand{\thefootnote}{\fnsymbol{footnote}}
\author{G. Last\footnotemark[1]\,, R. Szekli\footnotemark[2]\,
and D. Yogeshwaran\footnotemark[3]}
\footnotetext[1]{guenter.last@kit.edu, Karlsruhe Institute of
  Technology, Institute for Stochastics, 76131 Karlsruhe, Germany. }
\footnotetext[2]{szekli@math.uni.wroc.pl, University of Wroc{\l}aw, Mathematical Institute,
50-384, Wroc{\l}aw, Poland. Work supported by National Science Centre, Poland, grant NCN no 2015/19/B/ST1/01152}
\footnotetext[3]{d.yogesh@isibang.ac.in, Theoretical Statistics and
  Mathematics Unit, Indian Statistical Institute, Bangalore, India. }

\title{Some remarks on associated random fields, random measures and point processes}
\date{\today}
\maketitle
\begin{abstract}
\noindent  In this paper, we first show that for a countable family of random
  elements taking values in a partially ordered Polish space with a closed partial order (POP space),
  association (both positive and negative) of all finite dimensional
  marginals implies that of the infinite sequence. Our proof proceeds
  via Strassen's theorem for stochastic domination and thus avoids the
  assumption of normally ordered on the product space as needed for
  positive association in \cite{Lindqvist88}. 
  We use these results to show on POP spaces that finite dimensional negative
  association implies negative association of the random measure and
  negative association is preserved under weak convergence of random
  measures. The former provides a simpler proof in the most general
  setting of Polish spaces complementing the recent proofs in
  \cite{poinas2017mixing} and \cite{lyons2014determinantal} which
  restrict to point processes in $\R^d$ and locally compact Polish
  spaces respectively. We also provide some examples of associated
  random measures which shall illustrate our results as well.
\end{abstract}

\vspace{0.1cm}
\section{Introduction}
%early history
Positive association of random vectors in $\R^d$ appears in Esary et
al \cite{epw} in 1967, and negative association several years later, see
Joag-Dev and Proschan \cite{Joag-Dev1983} or Alam and Lai Saxena \cite{alam}. Since then the theory of
positive association has been well developed and has found many
applications in various contexts, for example to %%GL study
establish limit theorems,
to obtain concentration bounds or to derive stochastic comparison
results. Association of real random fields on $\Z^d$ and $\R^d$ were
used to obtain central limit theorems for
random fields, see e.g. Bulinski and Shashkin \cite{bulinski-sh}, Bulinski
and Spodarev \cite{bulinski-sp} or Poinas et al. \cite{poinas2017mixing} and references therein. We shall give examples of associated (positive and
negative) random fields later in Section \ref{sec:examples} and in the
Appendix. Positive association of probability measures on partially
ordered Polish spaces was studied by Lindqvist \cite{Lindqvist88},
where infinite products of such spaces and some space of
functions with values in partially ordered Polish spaces were
characterized by the corresponding finite dimensional distributions
under the additional assumption that the product space is normally
ordered.
%%Our first main result is to give another proof of this result
%%on positive association for infinite products of partially ordered
%%Polish spaces without this additional assumption of being normally
%%ordered.
Inspired by one of the proofs in Georgii and Yoo \cite{GeorYoo05}, we use
Strassen's theorem on stochastic domination to
prove the characterization by finite dimensional distributions for both positive and negative association for
countable families of random elements of general partially ordered Polish
spaces (see Theorem \ref{t2.1}), which generalizes Theorem 5.1 in Lindqvist \cite{Lindqvist88}. Using this idea, we characterize negative association by bounded, continuous, non-decreasing functions (Lemma \ref{t:suff_condn_NA}) and also show that association for countable families of random elements in partially ordered Polish spaces is preserved under weak convergence (Theorem \ref{t:weak_conv}).

A special case of partially ordered Polish spaces is that of the space
of all locally finite measures and in particular the space of locally
finite counting measures. These two spaces are of importance in the theory
of random measures and point processes.  Positive association of
random measures and point processes on locally compact Polish spaces
were characterized by the corresponding finite dimensional
distributions by Kwiecinski and Szekli \cite[Theorem
3.2]{kwiecinski-szekli96}. Using our Theorem \ref{t2.1}, we will prove
an analogous characterization (by finite dimensional distributions) of
negative association for random measures on Polish
spaces (Theorem \ref{measurena}). Similar results on negative
association for point processes on $\R^d$ and on locally compact
Polish spaces has been recently given by Poinas et al. \cite[Theorem
2.3]{poinas2017mixing} and Lyons \cite[paragraph
3.7]{lyons2014determinantal} respectively. Though the latter result is
in the context of determinantal point processes, the proof applies to
general negatively associated point processes. We will compare these
theorems in more detail in Section \ref{randommeasures}.  We will
extend these results into the context of random measures and will also
relax a rather restrictive assumption of local compactness on the
ground space. We use this along with Theorem \ref{t:weak_conv} to show
that weak convergence of random measures also preserves negative
association (Theorem \ref{t:NA_limit}). Apart from giving a very
general characterization of negatively associated random elements, our
result opens new possibilities, for example to obtain central limit
theorems for associated random measures in a quite general
context. Our results allow to extend several association properties of
countable random fields known only for finite dimensional
distributions into the setting of infinite sequences as well as
generate new examples of negatively associated random measures (see
Section \ref{sec:examples}).

% In \cite{kwiecinski-szekli96}, positive association of some
% continuous time processes from the theory of queues was shown
% (i.g. for virtual
% waiting %time processes). Another broad class of positively associated continuous time processes comes from the theory of interacting attractive particle %systems, see Liggett \cite{liggett}.

We end the introduction with a brief discussion %%on
of the theory of
negative dependence.  Though negative association was introduced in Joag-Dev and Proschan \cite{Joag-Dev1983} and Alam and Lai Saxena \cite{alam} in the context of reliability models,
it garnered significant additional interest following the article of Pemantle
\cite{Pemantle2000} in which he confined himself to binary-valued
random variables. The list of examples that motivated him to develop
techniques for proving that measures have negative dependence
properties such as negative association  or strong Rayleigh property include uniform random spanning trees, simple exclusion processes, random cluster models and
the occupation status of competing urns.  Among various definitions
expressing negative dependence, negative association seems to be one
of the easier conditions to verify and has also found
applications. Negative association has one distinct advantage over the
other types of negative dependence, namely, non-decreasing functions
of disjoint sets of negatively associated random variables are also
negatively associated. This closure property does not hold for the
other types of negative dependence. There exists nothing like a
general theory of negative association on partially ordered Polish
spaces, no reasonable analogy to the theory of positive association is
visible. New examples of negatively associated point processes and
random measures are given in Last and Szekli
\cite{last-szeklina} along with some stochastic comparisons of dependence. Positive association properties proved for many
interacting particle systems stay in contrast with the lack of
negative association results for most %%of
interacting particle
systems. A property related to negative association, known to hold for
symmetric exclusion processes is the strong Rayleigh property
(stronger than negative association) which is preserved in the
evolution of the symmetric exclusion process (see Borcea et al.
\cite[Theorem 5.2]{borceabl}).

The article is organized as follows. We introduce partially ordered
Polish spaces and Strassen's theorem in Section \ref{sec:prelims} and
then present our results about countable family of associated random
elements in Section \ref{sec:assoc_random_fields}. We then state and
prove our results on associated random measures in Section
\ref{randommeasures} and conclude with various (old and new) examples of associated
random elements in Section \ref{sec:examples}. At the end of this
paper, in an appendix, we present some additional examples which are
directly related to some applied stochastic models in order to gain a
broader view over this field. All formulations in the listed examples are in a strong sense PA and NA
as given in Definitions \ref{asso} and \ref{na-rm}. In many cases the exisiting results are known only for finite dimensional vectors but
we extend this to the infinite-dimensional vectors using our results.
%%%%%%%%%%%%%%%%%%%%%%%%%%%%%%%%%%%%%%%%%%%%%%%%

\section{Preliminaries}
\label{sec:prelims}

Let $\X$ be a Polish space endowed with a partial ordering
$\preceq$. A real-valued function $f$ on $\X$ is said to be
non-decreasing if $x\preceq y$ implies $f(x)\le f(y)$. We let $\cX $
denote the Borel $\sigma$-field on $\X$. 
%as well as the class of
%real-valued functions measurable w.r.t.\ that $\sigma$-field. 
For
probability measures $P$ and $P'$ on $(\X, \cX)$, $P$ is
stochastically dominated by $P'$ if
$$
\int f dP \le  \int fdP'
$$
for all non-decreasing bounded measurable $f$. %%We write that
In this case we write
$P\preceq_{st}P'$. We assume that the partial ordering $\preceq$ is
closed, i.e., the set $H = \{(x, x'); x \preceq x'\}$ is closed in the
product topology on $\X^2$.
%%A classical (Strassen's) theorem on
%%stochastic domination states that
For the reader's convenience we state
the classical (Strassen) theorem on stochastic domination:

\begin{theorem}\label{strassen}
  $P$ and $P'$ satisfy $P\preceq_{st}P'$ iff there exists a
  probability measure $\tilde P$ on $(\X^2, \cX^2)$ with marginals $P$
  and $P'$ such that $\tilde P(H)= 1$.
\end{theorem}
This result is often referred to as ``Strassen's Theorem", which is
formally misleading: in Strassen \cite{strassen} it is only briefly
mentioned as one possible application of Theorem 11 in that paper, and
the condition $P\preceq_{st}P'$ does not appear explicitly. An
explicit formulation can be found in \cite[Theorem 1]{kamaeko}. For a nice proof of
Theorem \ref{strassen} and some additional observations, see Lindvall
\cite{lindvall}. It is known (see \cite[Theorem 2]{kamaek}) that the
relation $\preceq_{st}$ on the space of probability measures on
$(\X, \cX)$ with the topology of weak convergence is a closed partial
ordering.
% In \cite{kamaeko}, Proposition 2, the special case of product spaces
% is considered. For $\X=\prod_{i=1}^\infty \X_i$ where $\X_i$ are
% Polish spaces $\X$ is also Polish with the product topology.
%\begin{theorem}\label{kamaeko}
%  If $P, P'$ are probability measures on $(\X,\cX)$, and
%  $P^{(n)}, P'^{(n)}$ denote the respective final dimensional
%  distributions on $\prod_{i=1}^n\X_i$ then
%  $P^{(n)}\preceq_{st}P'^{(n)}$ for all $n\ge 1$ implies
%  $P\preceq_{st}P'$.
%\end{theorem}
%%%%%%%%%%%%%%%%%%%%%%%%%%%%%%%%%%%%%%%%%%%%%
\section{Association of discrete random fields}
\label{sec:assoc_random_fields}

Let $I$ be countable index set (e.g.\ $I=\{1,\ldots,n\}$, $I=\Z^d$, $I=\N$).  Let ${\mathbf X}=(X_i)_{i\in I}$ be a random field, that
is a family of random elements with values
in a partially ordered Polish (POP) space $(\X,\cX)$.
%%in a partially ordered
%%Polish (POP) space $(\X,\cX)$.
For $J\subset I$, we write $X_J:=(X_i)_{i\in J}$.

\begin{definition}\label{asso}\rm
For a family ${\mathbf X}=(X_i)_{i\in I}$ of random elements of $(\X, \cX)$
\begin{itemize}
\item[(i)]
we say that ${\mathbf X}$ is negatively associated
(\na) if
\begin{align}\label{NA}
\BE[f(X_J)g(X_{J^c})]\le \BE[f(X_J)]\BE[g(X_{J^c})]
\end{align}
for all $J\subset I$ and for all (coordinatewise) non-decreasing bounded
measurable $f\colon \R^J\to\R$ and $g\colon \R^{J^c}\to\R$
for which the expectations in \eqref{NA} exist,  where $\R^J$ denotes the space of all real functions defined on $J$;
\item[(ii)]
we say that ${\mathbf X}$ is positively associated
(\pa) if
\begin{align}\label{PA}
\BE[f(X_J)g(X_{J})]\ge \BE[f(X_J)]\BE[g(X_{J})]
\end{align}
for all $J\subset I$ and for all (coordinatewise) non-decreasing bounded
measurable $f\colon \R^J\to\R$ and $g\colon \R^{J}\to\R$
for which the expectations in \eqref{PA} exist, where $\R^J$ denotes the space of all real functions defined on $J$.
\end{itemize}
\end{definition}
\begin{remark}\rm\label{rem1} Association for uncountable index sets $I$ can be defined as follows.
For a family ${\mathbf X}=(X_i)_{i\in I}$ of random elements of
  $(\X, \cX)$ we say that ${\mathbf X}$ is positively associated (\pa)
  if
\begin{align}\label{UCPA}
\BE[f(X_J)g(X_{J})]\ge \BE[f(X_J)]\BE[g(X_{J})]
\end{align}
for all countable $J\subset I$ and for all (coordinatewise)
non-decreasing bounded measurable $f\colon \R^J\to\R$ and
$g\colon \R^{J}\to\R$ for which the expectations in \eqref{UCPA} exist.
%where $\R^J$ denotes the space of all real functions defined on $J$.
Similarly, one can define \na \, using disjoint countable index sets
$J,J'\subset I$.
\end{remark}
%We say that ${\mathbf X}$ is {\it weakly positively associated} {\bf DO WE NEED THIS NOTION ?}
%if
%\begin{align}\label{wPA}
%\BE[f(X_J)g(X_{J})]\ge \BE[f(X_J)]\BE[g(X_{J})]
%\end{align}
%for all finite $J\subset I$ and for all (coordinatewise) non-decreasing bounded
%measurable $f\colon \R^{|J|}\to\R$ and $g\colon \R^{|J|}\to\R$
%for which the expectations in \eqref{PA} exist, where $|J|$ denotes the cardinality of $J$.

By the well-known formula
\begin{align}
\label{e:covid}
\BC[X,Y] & =\int \BC[\I\{X>s\},\I\{Y>t\}]\,d(s,t),
\end{align}
valid for all integrable random variables $X$ and $Y$ with
$\BE[|XY|]<\infty$, it is enough to assume in $\eqref{NA}$ that $f$
and $g$ are non-negative. The above identity can be found in the proof
of \cite[Lemma 2]{Lehmann66} which the author attributes to
\cite{Hoeffding40}.

We say that a family ${\mathbf X}=(X_i)_{i\in I}$ of random elements of $(\X, \cX)$ is associated if it is \pa \, or \na \,. In our proofs we will concentrate on the \na \, case, only pointing out how to deal with the \na \, case.

We now state and prove one of our main theorems showing
that NA property of finite dimensional marginals implies that of the
infinite sequence. Our proof was inspired by the proof
of \cite[Corollary 3.4]{GeorYoo05}.
\begin{theorem}
\label{t2.1}
Consider a discrete family ${\mathbf X}=(X_n)_{n\in I}$ of random
elements of POP space $(\X, \cX)$. Assume that for each finite
$J\subset I$ the finite subfamily $X_J$ is associated . Then the family
${\bf X}$ is associated in the same positive or negative way as finite subfamilies.
\end{theorem}%
\noindent {\em Proof:} We prove the \na \, case. In order to check \eqref{NA}, let us first
assume that $J\subset I$ is finite and because of \eqref{e:covid}, let
$f\colon \R^J\to[0,\infty)$ be non-decreasing and such that $\BE[f(X_J)]<\infty$. It is
no restriction of generality to assume that
$\BE[f(X_J)]>0$. (Otherwise we have that $\BP(f(X_J)=0)=1$ and
\eqref{NA} becomes trivial.) Since we assumed that $I$ is discrete, we
can enumerate elements of $I$ and assume that $J=\{1,\ldots,m\}$ for
some $m\in\N$.  For $n\in\N$, we define a random element
$X^{(n)}_J=(X^{(n)}_k)_{k\ge 1}$ of $\X^\N$ by $X^{(n)}_k:=X_{m+k}$
for $k\in\{1,\ldots,n\}$ and $X^n_k:=z$ for $k\notin\{1,\ldots,n\},$
for a fixed element $z\in \X$.  By our assumption, we have that for
all $n\in \N$
\begin{align*}
\BE[f(X_J)g(X^{(n)}_J)]\le \BE[f(X_J)]\BE[g(X^{(n)}_J)]
\end{align*}
for all measurable non-decreasing $g\colon \X^\N\to\R$
such that $\BE[|g(X_J^{(n)})|]<\infty$. But this means
that
\begin{align}\label{st}
\mu_{n,J}\preceq_{st}\nu_{n,J},\quad n\in\N,
\end{align}
where $\mu_{n,J}:=\BE[f(X_J)]^{-1}\BE[f(X_J)\I\{X^{(n)}_J\in\cdot\}]$,
$\nu_{n,J}:=\BP(X^{(n)}_J\in\cdot)$ and $\preceq_{st}$ denotes
strong stochastic ordering of probability measures on $\X^\N$
(w.r.t.\ coordinatewise $\preceq $ ordering). Moreover, the set
$$
H:=\{(x,y)\in \X^\N\times\X^\N:x\preceq y\}
$$
is closed w.r.t.\ the product topology on $\X^\N\times\X^\N$.
By Strassen's theorem there exists for each $n\in\N$ a
probability measure $\gamma_n$ on $\X^\N\times\X^\N$ with
marginals $\mu_n$ and $\nu_n$, respectively, such that
$\gamma_n(H)=1$.

By \cite[Theorem 4.29]{Kallenberg02}
we have that $\mu_{n,J}\overset{d}{\to}\mu_J$
as $n\to\infty$,
where
\begin{align*}
\mu_J:=\BE[f(X_J)]^{-1}\BE[f(X_J)\I\{X_J^\infty\in\cdot\}].
\end{align*}
Similarly, $\nu_{n,J}\overset{d}{\to}\nu_J:=\BP(X_J^\infty\in\cdot)$,
where $X^\infty_k:=X_{m+k}$, $k\in\N$.

Now we use a similar argument as in \cite[Proposition 3]{kamaeko}.
By  \cite[Theorem 16.3]{Kallenberg02} we have  that
the sequences $(\mu_n)$ and $(\nu_n)$ are tight.
Since $\gamma_n$ has marginals $\mu_n$ and $\nu_n$ we have
for any measurable $A,B\subset\X^\N$ that
\begin{align*}
\gamma_n((A\times B)^c)=\gamma_n(A^c\times B)+\gamma_n(A\times B^c)+\gamma_n(A^c\times B^c)
\le 2\mu_n(A^c)+\nu_n(B^c).
\end{align*}
Therefore the sequence $(\gamma_n)$ is also tight. Let
$\gamma$ be a subsequential limit.
Since $H$ is closed, the Portmanteau theorem shows
that $\gamma(H)=1$. By definition of weak convergence,
$\gamma$ has marginals $\mu$ and $\nu$, respectively.
But this implies that $\int g\,d\mu \le\int g\,d\nu$
for all measurable non-decreasing bounded $g\colon\X^\N\to[0,\infty)$,
so that \eqref{NA} follows.

%\bigskip
%{\bf Remark:  "it is easy to see" in the above paragraph has been replaced by a more detailed argumentation.}
%\bigskip

Finally we take an arbitrary $J\subset\N$. By the first step of the
proof we have
\begin{align*}
\BE[f(X_J)g(X_{J'})]\le \BE[f(X_J)]\BE[g(X_{J'})]
\end{align*}
for all finite sets $J'\subset\N\setminus J$, for all non-decreasing
measurable bounded $f\colon \X^J\to\R$ and $g\colon \X^{J'}\to\R$.
Repeating the above arguments yields  \eqref{NA} in full generality.

The proof for the PA case can be done in a similar way.
\qed

\bigskip

Our proof technique gives the following  corollary.
\begin{corollary}
\label{c:NA_2_Seq}
Suppose ${\mathbf X}=(X_n)_{n\in I}$ and ${\mathbf Y}= (Y_n)_{n\in J}$
are two discrete families of random elements of POP space $(\X,
\cX)$.
Assume that for all finite $I',J'$ and non-decreasing bounded
measurable $f,g$, it holds that
$$ \BE[f(X_{I'})g(Y_{J'})] \leq ( \geq )\BE[f(X_{I'})]\BE[g(Y_{J'})].$$
Then, for all non-decreasing bounded measurable $f,g$ and countable $I$, $J$, we have that
$$ \BE[f(X_{I})g(Y_{J})] \leq (\geq)\BE[f(X_{I})]\BE[g(Y_{J})].$$
\end{corollary}
\remove{Theorem \ref{t2.1} can be reformulated, with a similar proof, for \pa,
and as such would be an extension of the theorem proved by
Linqvist \cite[Theorem 5.1]{Lindqvist88}.

\begin{theorem}\label{t.lind} Consider a discrete family
  ${\mathbf X}=(X_n)_{n\in I}$ of random elements of POP space
  $(\X, \cX)$. Assume that for each finite $J\subset I$ the finite
  subfamily $X_J$ is
  \pa . Then the family ${\bf X}$ is \pa.
\end{theorem}
}
A second very useful consequence of our proof technique is the following lemma allowing us to restrict \eqref{NA} to only bounded continuous non-decreasing functions. 
\begin{lemma}
\label{t:suff_condn_NA}
Consider a finite family $\mathbf X = (X_1,\ldots, X_m)$ such that it satisfies \eqref{NA} (\eqref{PA}) for all non-negative, bounded, continuous, non-decreasing functions $f,g$ on $\X^J,\X^K$ respectively where $J \subset \{1,\ldots,m\}$ and $K = \{1,\ldots,m\} - J$. Then $\mathbf X$ is $\na \,( \pa \,)$.
\end{lemma}
\begin{proof}
We shall again prove in the case of \na \, alone. Let $J,K$ be as assumed in the lemma. From \eqref{e:covid}, it suffices to show \eqref{NA} for all non-negative bounded, measurable non-decreasing functions. Let $g$ be a non-negative, bounded, continuous, non-decreasing function such that $\BE[g(X_K)] > 0$. Thus, we have for all non-negative, bounded, continuous, non-decreasing functions $f$ that
\begin{align*}
\BE[f(X_J)g(X_{K})]\le \BE[f(X_J)]\BE[g(X_{K})],
\end{align*}
and this inequality can be re-written in the form of a stochastic order relation as
\begin{align*}
\BE[g(X_K]^{-1}\BE[f(X_J)g(X_K)]\le \BE[f(X_J)].
\end{align*}
Defining probability measures $\nu_J := \BP(X_J \in \cdot)$, $\mu^g_{J} := \BE[g(X_K)]^{-1}\BE[g(X_K)\I\{X_J \in\cdot\}]$, we have that the above inequality implies $\mu^g_J \preceq_{st} \nu_J$ by \cite[Theorem 2.6.4]{mullercomparison}. From the definition of $\preceq_{st}$ (i.e., stochastic domination), we have that 
\begin{align*}
\BE[f(X_J)g(X_{K})]\le \BE[f(X_J)]\BE[g(X_{K})],
\end{align*}
for all non-negative, bounded, measurable functions $f$. Now repeating the above argument by fixing a non-negative bounded measurable function $f$ such that $\BE[f(X_J)] > 0$, we can derive that \eqref{NA} holds for all non-negative, bounded, measurable non-decreasing functions $f,g$ as required to complete the proof. 
\end{proof}
A powerful consequence of the above lemma is that the property of association is preserved under weak convergence. We shall use this theorem in our next section on random measures but only in the case of $\X = \R$.  
\begin{theorem}\label{t:weak_conv}
  For $k \geq 1$, consider a discrete family ${\mathbf X}^k=(X^k_i)_{i \in I}$ of random elements of POP space $(\X, \cX)$. Assume
  that $\mathbf X^k$ is associated for every $k \geq 1$ in the same way (i.e., always $\pa \,$ or always $\na \,$) and
  $\mathbf X^k \stackrel{d}{\to} \mathbf X$ as $k\to \infty$. Then, $\mathbf X$ is associated in the same positive or negative way as the elements in the sequence.
\end{theorem}
\begin{proof}
From our assumptions, we have that for each $k\ge 1$, $m\ge 1$,  $(X^k_1,\ldots, X^k_m) $ is associated, and  $(X^k_1,\ldots, X^k_m)\stackrel{d}{\to} (X_1,\ldots, X_m)$. Thus, we have that $(X_1,\ldots, X_m)$ satisfies \eqref{NA} (or \eqref{PA}) for for all non-negative, bounded, continuous, non-decreasing functions $f,g$ defined on disjoint index sets of $\{1,\ldots,m\}$. Now, from Lemma \ref{t:suff_condn_NA} we have that $(X_1,\ldots,X_m)$ is a finite \na \, (or \pa \,) family and because of our Theorem \ref{t2.1}, this suffices to conclude that $\mathbf X$ is a \na \,(or \pa \,) family. 
\end{proof}
We now compare our above results and proof techniques to those in the literature. Under the assumption that the product POP space is normally ordered, Lemma \ref{t:suff_condn_NA} and Theorem \ref{t:weak_conv} are shown  for \pa \, in \cite[Theorem 3.1(v)]{Lindqvist88}. Lemma \ref{t:suff_condn_NA} for \pa \ is shown for $\X = \R$ in \cite[Lemma 3.1 and Theorem 3.3]{epw}. The proof techniques of \cite{epw} and \cite{Lindqvist88} involve approximating binary, non-decreasing, measurable functions by non-negative, bounded, continuous, non-decreasing functions and these require additional assumptions on the space $\X$ relating the metric and order. These ideas can also be implemented in the case of \na \, with suitable modifications. However, our proof avoids these by using Strassen's theorem and similar criteria holding for stochastic domination.

An alternative assumption to normally ordered spaces is the following condition formulated in \cite{ludger} (recalled as (R1) in \cite{Lindqvist88}) : $x \mapsto d(x,A)$ is non-increasing for an increasing set $A$. Under this assumption, the proof ideas as in \cite[Lemma 3.1 and Theorem 3.3]{epw} or \cite[Theorem 3.1(v)]{Lindqvist88} or \cite[Theorem 1(d)]{ludger} can be adapted suitably for both \pa \, and \na \,. Also, we would like to mention that this condition and the property of being normally ordered need not be related (see \cite[pg. 38]{noebels}). 
%
%
%\begin{proof}
  %We prove the NA case. By Theorem \ref{t2.1}, it suffices to show that for all
  %$S \subset T$ finite, $X_S := (X_n)_{n \in S}$ has NA-property. Let
  %us consider the sequence of vectors $X^k_S := (X^k_n)_{n \in S}$.
  %Let $S_1,S_2$ be disjoint subsets of $S$ and $f_i\colon \X^{S_i} \to \R$
  %be bounded, continuous and non-decreasing function. Then, by
  %NA-property of $X^k_S$, we have that
  %$\BC(f_1(X^k_{S_1}),f_2(X^k_{S_2})) \leq 0$ for all $k \geq 1$ and
  %since $X^k_S \stackrel{d}{\to} X_S$, we also have that
  %$\BC(f_1(X^k_{S_1}),f_2(X^k_{S_2})) \to
  %\BC(f_1(X_{S_1}),f_2(X_{S_2}))$
  %as $f_1,f_2$ are bounded, continuous functions. Now this verifies
  %(ii) of Theorem \ref{t:suff_condn_NA} and hence $X_S$ has
  %NA-property as required. The PA case is similar.
%\end{proof}
%
\section{Association of random measures}
\label{randommeasures}

Let $\sS$ be a Polish space, $\cS$ be the $\sigma$-field of Borel
subsets of $\sS$, and $\cS_b$ be the ring of bounded Bore1 sets in
$\sS$. By a random measure $M$ on $\sS$ we mean a mapping of some
probability space $(\Omega, {\cal F}, P)$ into the space $\mathbf{M}(\sS)$ of
Radon measures on $(\sS,\cS)$, equipped with the smallest $\sigma$-field making
the mappings $\mu\mapsto\mu(B)$ measurable for all $B\in \cS$.
When $M$ is a.s.\ confined to the space
$\mathbf{N}(\sS)\subset\mathbf{M}(\sS)$
of integer valued measures, we say that $M$ is a
point process.  Vague convergence
$\mu_n\to \mu$ in $\mathbf{M}(\sS)$ means that
$$
\int_\sS f d\mu_n\to \int_\sS f d\mu
$$
for each continuous $f\colon\sS\to \R_+$ with
bounded support.  A natural partial ordering on $\mathbf{M}(\sS)$ and
$\mathbf{N}(\sS)$ is given by: $\mu < \nu$ if $\mu (B)\le \nu (B),$ for all
$B\in \cS_b$.  It is known \cite[Lemma 1]{rolski-szekli} that the
vague topology and the partial order $< $ are related, namely $<$ is
closed, i.e.\ the set $\{(\mu, \nu): \mu <\nu\}\subset \mathbf{M}(\sS)^2$ is closed
in the product topology on $\sS^2$.
%Since the space $\mathbf{M}(\sS)$ with the
%vague topology is Polish and the partial order $<$ is closed we can
%define the corresponding strong stochastic ordering for random
%elements of $\mathbf{M}(\sS)$.
We denote the strong stochastic ordering of random elements
of $\mathbf{M}(\sS)$ by $<_{st}$.  A random
measure $M$ is then said to be positively associated (\pa )\
%%if and only if
\begin{align}\label{MPA}
\BE[f(M)g(M)]\ge \BE[f(M)]\BE[g(M)]
\end{align}
for any pair of real valued, bounded measurable functions $f, g$ on $\mathbf{M}(\sS)$, non-decreasing w.r.t.\ the order $<$.

Let ${\cal I} \subset \cS_b$ be a countable, topological,
dissecting, semi-ring generating
the $\sigma$-field $\cS$, as defined in \cite[Lemma 1.9]{kallenberg17}.
Denote by $I_1, I_2,\ldots $ some enumeration of the elements of $\mathcal{I}$.
Define the mapping $\gamma\colon\mathbf{M}(\sS)\to \R_+^\infty$ by
\begin{equation}\label{one-one}
\gamma (\mu):=(\mu (I_1), \mu (I_2),\ldots )
\end{equation}
and let $\mathbb{G}:=\gamma (\sS)$.  Since $\cal I$ is a semi-ring
generating $\cS$, by \cite[Theorem 11.3]{billingsley} the mapping
$\gamma$ is 1-1 and it is also increasing. Let $\rho$ be a complete metric in $\sS$ generating
the vague topology. Define a metric $\rho_\gamma$ in $\mathbb{G}$ by
$$
\rho_\gamma (x,y) =\rho (\gamma ^{-1}(x),\gamma^{-1}(y)),
$$
for all $x,y\in \mathbb{G}$. We recall some basic properties of $\mathbb{G}$;
see \cite[Lemma 2]{rolski-szekli} and \cite[Theorem A1.3]{Kallenberg02}.

\begin{lemma}\label{lemmar-s}
\begin{itemize}
\item[(i)] We have that
$\mathbb{G}\in {\cal B}(\R_+^\infty)$ and that the inverse map
$\gamma^{-1}\colon \mathbb{G}\to\mathbf{M}(\sS)$ is measurable.
\item[(ii)]
$\mathbb{G}$ is metrizable as a Polish space by the metric $\rho_\gamma$.
\item[(iii)] The Borel $\sigma$-field ${\cal B}(\mathbb{G})$
  generated by $\rho_{\gamma}$ is of the form
  ${\cal B}(\mathbb{G})=\mathbb{G}\cap {\cal B}(\R_+^\infty).$
\end{itemize}
\end{lemma}
%
%\begin{theorem}\label{measurepa}
%Let M be a random measure on a  locally  compact  Polish  space $\sS$. Then  M is positively associated if and only if random
%vectors $(M(B_1),\ldots M(B_n))$ are positively associated for all $n \ge 1,$ and bounded sets $B_1,\ldots,B_n\in \cS_b$.
%\end{theorem}
For a Borel set $A\subset \sS$, let ${\cal F}(A)$ denote the
$\sigma$-field on $\mathbf{N}(\sS)$ generated by the functions
$\mu \mapsto \mu (B)$ for Borel $B\subseteq A$. We say that a function
on $\mathbf{N}(\sS)$ is measurable with respect to $A$ if it is
measurable with respect to ${\cal F}(A)$. For each measure $\mu$ on
$\sS$, we denote by $\mu_A:=\mu(\cdot\cap A)$ the restriction of $\mu$
to $A$. Then a measurable function $f\colon \mathbf{M}(\sS)\to\R$ is
$A$-measurable iff $f(\mu)=f(\mu_A)$ for each $\mu\in\mathbf{M}(\sS)$.

The following definition is an extension to random measures of definitions used by  Lyons \cite{lyons2014determinantal}  and Poinas et al. \cite{poinas2017mixing} for point processes.
\begin{definition}\label{na-rm}\rm
We say that a random measure $M$ is  negatively associated (NA) if
\begin{align}\label{MNA}
\BE[f(M)g(M)]\le \BE[f(M)]\BE[g(M)],
\end{align}
for every pair $f, g$ of bounded non-decreasing functions that are
measurable with respect to disjoint measurable subsets of $\sS$.
\end{definition}
\begin{remark}\rm\label{rem2}
The above definition of NA property for random measures is not equivalent to the one given in Remark \ref{rem1} when random measures are viewed as random fields indexed by the uncountable set $\{B : B \in \cS\}$. But for PA property, these two definitions - \eqref{MPA} and that in Remark \ref{rem1} - are equivalent.
\end{remark}
We shall again refer to a random measure as associated if it is either negatively associated
or positively associated.  As a consequence of Lemma \ref{lemmar-s}, Kwiecinski and Szekli
\cite[Theorem 3.2]{kwiecinski-szekli96} proved for locally compact
spaces that the random measure $M$ is positively associated iff random
vectors $(M(B_1),\ldots M(B_n))$ are positively associated for all
$n \ge 1,$ and bounded sets $B_1,\ldots,B_n\in \cS_b$.  We next show
an analogous result for the NA-property. We shall relax the assumption
on local compactness. To get the positive association result it was
enough to use the fact that non-decreasing transformations of
positively associated random elements into another partially ordered
space are again positively associated elements of this space. For
negative association this property does not hold. In Poinas et al.
\cite[Theorem 2.3]{poinas2017mixing} the proof of an analog of Theorem
\ref{measurena} is given for point processes on $\sS=\R^d$. They use a
variant of the monotone class theorem. This proof is not directly
applicable for more general spaces $\sS$. A proof of the NA part of Theorem \ref{measurena} for
point processes on locally compact partially ordered Polish spaces
can be (implicitly) found in Lyons \cite[paragraph
3.7]{lyons2014determinantal}, where negative association of some
determinantal point processes on locally compact Polish spaces is
proved. The arguments there are rather lengthy and are based on
Lusin's separation theorem and the Choquet capacitability theorem. We
shall give a short proof of this result in a more general setting of
random measures, using Lemma \ref{lemmar-s} and Theorem \ref{t2.1}.
\begin{theorem}\label{measurena}
  Let M be a random measure on a Polish space $\sS$. Then M is
  associated if and only if random vectors
  $(M(B_1),\ldots M(B_n))$ are associated in the same positive or negative way for all
  $n \ge 1,$ and disjoint sets $B_1,\ldots,B_n\in {\cal I}$.
\end{theorem}
Before proving the theorem, we need a lemma that will allow us to
assume that the bounded disjoint sets can be taken to be measurable in
the above theorem instead of just elements of $\mathcal{I}$.
\begin{lemma}\label{lem:meas-ring-NA}
Let $M$ be a random measure on a Polish space $\sS$. Then
$(M(B_1),\ldots, M(B_n))$ is %%negatively associated
associated for all $n \ge 1,$
and disjoint sets $B_1,\ldots,B_n\in {\cal I}$ iff
$(M(B_1),\ldots, M(B_n))$ is associated in the same positive or negative way
for all $n \ge 1$
and disjoint sets $B_1,\ldots,B_n \in \cS_b$.
\end{lemma}
{\em Proof:} We shall again prove only for NA property and the same proof applies more easily to PA.
The `if' part is trivial as $\mathcal{I} \subset \cS_b$ and we shall now prove the other part.
Fix $m$ and disjoint $B_2,\ldots,B_m \in \mathcal{I}$ and
consider the class $\mathcal{M}$ of all bounded measurable sets $B$
such that $(M(B\setminus (B_2\cup\ldots \cup B_m)),M(B_2),\ldots, M(B_m))$
is NA. If $B \in \mathcal{I}$, then $B\setminus (B_2\cup\ldots \cup B_m)$
can be written as a finite disjoint union of $\mathcal{I}$-sets and
hence $(M(B\setminus (B_2\cup\ldots \cup B_m)),M(B_2),\ldots, M(B_m))$ is
NA. So, $\mathcal{I} \subset \mathcal{M}$. Denoting by
$\mathcal{R}(\mathcal{I})$, the ring generated by taking finite unions
of sets in $\mathcal{I}$, we have that
$\mathcal{R}(\mathcal{I}) \subset \mathcal{M}$. Further, by Theorem 
\ref{t:weak_conv}, we have that $\mathcal{M}$ is closed under bounded
monotone limits and so $\mathcal{M}$ is a local monotone class. By the
(local) monotone class theorem (\cite[Lemma 1.2]{kallenberg17}),
$\mathcal{M}$ contains the local monotone ring generated by the ring
$\mathcal{R}(\mathcal{I})$ which is nothing but $\cS_b$. Hence
$(M(C),M(B_2),\ldots,M(B_m))$ is NA for all $C \in \cS_b$ such that $C$
is disjoint from $B_2,\ldots,B_m$. Repeating this argument, we can
derive the asserted property for all disjoint
$B_1,\ldots,B_m \in \cS_b$. \qed

\bigskip

{\em Proof of Theorem \ref{measurena}:} Again, we shall prove only for NA and the proof for the PA case follows similarly.
The `only if' part is trivial and so we shall prove the
'if' part. Fix a pair $f, g$ of bounded non-decreasing functions that
are measurable with respect to disjoint measurable subsets of $\sS$,
say $A$, $B$.
%We shall re-index the DC-semiring so that
%$I_j \subset A$ only if $j$ is odd and $I_j \subset B$ only if $j$ is
%even.
Using \eqref{one-one}, define on $\mathbb{G}$ two measurable functions
$\tilde f:=f\circ \gamma ^{-1}$, and
$\tilde g:=g\circ \gamma ^{-1}$.
It is not hard to prove (by a monotone class
argument for instance) that $\gamma^{-1}$ is non-decreasing, so that
$\tilde f$ and $\tilde g$ are non-decreasing.
Define $X^A:=(M(I_n\cap A))_{n\ge 1}=\gamma(M_A)$ and $X^B:=(M(I_n\cap B))_{n\ge 1}$.
Suppose we can show that
\begin{align}\label{e4.5}
\BE[\tilde f(X^A)\tilde g(X^B)]\le \BE[\tilde f(X^A)]\BE[\tilde g(X^B)].
\end{align}
Then we would obtain that
\begin{align*}
\BE[f(M)g(M)]&=\BE[f(M_A)g(M_B)]=\BE[\tilde f(X^A)\tilde g(X^B)]\\
&\le \BE[\tilde f(X^A)\BE[\tilde g(X^B)]=\BE[f(M)]\BE[g(M)],
\end{align*}
as desired.

It remains to prove \eqref{e4.5} for arbitrary bounded non-decreasing
measurable functions $\tilde f$ and $\tilde g$.
To this end, we take $m\in\N$ and show that
\begin{align}\label{e4.7}
\BE[h_1(M(I_k\cap A)^m_{k=1})h_2(M(I_k\cap B)^m_{k=1})]
\le \BE[h_1(M(I_k\cap A)^m_{k=1})]\BE[h_2(M(I_k\cap B)^m_{k=1})],
\end{align}
for all bounded non-decreasing
measurable functions $h_1\colon\R_+^m\to\R$ and
$h_2\colon\R_+^m\to\R$. There exist $l\in\N$ and
disjoint sets $I'_i \in \mathcal{I}$, $i\in\{1,\ldots,l\}$, such
that $I_j = \cup_{i \in J_j}I'_i$ for $j\in\{1,\ldots,m\}$,
where $J_j \subset \{1,\ldots,l\}$ for all $j$.
Defining
\begin{align*}
h'_1((x_i)^l_{i=1})
:= h_1\bigg(\sum_{i \in J_1}x_i,\ldots,\sum_{i \in J_m} x_i\bigg),
\end{align*}
we observe that $h'_1$ is coordinatewise non-decreasing as $h_1$ is
coordinatewise non-decreasing. Similarly, we can define $h'_2$.
By disjointness of $I'_1,\ldots,I'_l$ and $A\cap B=\emptyset$,
we have by assumption and Lemma \ref{lem:meas-ring-NA} that the random vector
$$
(M(I'_1\cap A),\ldots,M(I'_l\cap A),M(I'_1\cap B),\ldots,M(I'_l\cap B))
$$
is negatively associated. Therefore
\begin{align*}
\BE[h'_1((M(I'_i\cap A)^l_{i=1})h'_2((M(I'_i\cap B)^l_{i=1})]
\le \BE[h'_1((M(I'_i\cap A)^l_{i=1})]\BE[h'_2((M(I'_i\cap B)^l_{i=1})].
\end{align*}
By definition of $h'_1,h'_2$, the above inequality is equivalent to
the inequality \eqref{e4.7}. Now using Corollary \ref{c:NA_2_Seq}, we
obtain \eqref{e4.5} as required to complete the proof. \qed

\bigskip

We use $\overset{d}{\to}$ to denote weak convergence of random measures as well.
\begin{theorem}\label{t:NA_limit}
  Suppose $M_n, n \geq 1$ are associated random measures on a Polish space $\sS$
  and $M_n \stackrel{d}{\to} M$. Then $M$ is also associated as a random measure in the same positive or negative way as the elements of the sequence.
\end{theorem}
\begin{proof}
As before, we prove only the NA case and the PA case follows analogously.
Define $\cS_M := \{ B \in \cS_b : \BE[M(\partial B)] = 0 \}$ where
  $\partial B$ is the boundary of a set $B$. Since
  $M_n \stackrel{d}{\to} M$, we have that
  $(M_n(B_1),\ldots,M_n(B_k)) \stackrel{d}{\to}
  (M(B_1),\ldots,M(B_k))$
  for all $B_1,\ldots,B_k \in \cS_M$ (see \cite[Theorem
  4.11]{kallenberg17}).  Thus by Theorem \ref{t:weak_conv}, we have
  that $(M(B_1),\ldots,M(B_k))$ is NA for all pairwise disjoint
  $B_1,\ldots,B_k \in \cS_M$.  Since $\cS_M$ is a dissecting ring and
  there exists ${\cal I} \subset \cS_M$, a countable, topological,
  dissecting, semi-ring generating the $\sigma$-field
  $\cS$ (\cite[Lemma 1.9]{kallenberg17}), by Theorem \ref{measurena},
  we have that $M$ is a NA random measure.
\end{proof}
%%%%%%%%%%%%%%%%%%%%%%%%%%%%%%%%%
%%%%%%%%%%%%%%%%%%%%%%%%%%%%%%%%%%
\section{Examples}
\label{sec:examples}
In this section, we recall some known and give some new examples of
associated random measures and fields. As mentioned in the introduction, showing many of these examples are associated in the strong sense as in Definitions \ref{asso}, \ref{na-rm} and \eqref{MPA} shall require our Theorems \ref{t2.1}, \ref{t:weak_conv}, \ref{measurena} and \ref{t:NA_limit}. In an appendix, we recall some classical results related directly to applied probability models. We are not aware of many examples of NA random fields.
%%%%%%%%%%%%%%%%%%%%%%%%%%%%%%%%%%
\subsection{Associated random fields}
%\label{sec:NA_ex}
%%%%%%%%%%%%%%%%%%%%%%%%%%%%%%%%%%
\begin{example}\rm (Gaussian random measures and fields.) Suppose that
  $M$ is a Gaussian random measure on $\mathbb{S}$ such that
  $\BC[M(A),M(B)] \leq 0$ for $A$ and $B$ disjoint. Then from
  \cite[Section 3.4]{Joag-Dev1983} and our Theorem \ref{t2.1}, we have
  that $M$ is  NA. A simple special case is when
  $\mathbb{S}$ is a discrete set and $M:= \sum_{s\in\mathbb{S}}X_s\delta_s$
  where ${\mathbf X} := (X_s)_{s \in \mathbb{S}}$ is a Gaussian random
  field such that $\BC[X_s,X_t] \leq 0$ for all $s \neq t$, which implies that ${\mathbf X}$ is NA. Similarly,
  by \cite{pitt} and our Theorem \ref{t2.1}, the condition
  $\BC (X_s,X_t)\ge 0$, $s,t \in \mathbb{S}$, is necessary and sufficient for
  the random measure $M$ to be PA.
\end{example}
%%%%%%%%%%%%%%%%%%%
\begin{example}\label{dirichletseq}\rm (Dirichlet sequences) Let
$\alpha_n\ge 0$, $n\in\N$, be such that $\alpha:=\sum^\infty_{n=1}\alpha_n$
is positive and finite.
Let $X_1,X_2,\ldots$ be independent Gamma distributed random
variables with shape parameters $\alpha_1,\alpha_2,\ldots$ and scale parameter $1$.
Then $X:=\sum^\infty_{n=1}X_n$ has a Gamma distribution with shape parameter
$\alpha$ and $(X^{-1}X_n)_{n\ge 1}$ is NA. To see the latter we first assume that
there exists $m\in\N$ such that $\alpha_n=0$ for $n> m$.
Then $X^{-1}(X_1,\ldots,X_m)$ has a Dirichlet distribution with parameter
$(\alpha_1,\ldots,\alpha_m)$. Moreover, since the latter random vector
is independent of $X$, we obtain from \cite[Theorem 2.8]{Joag-Dev1983}
that it is NA. Since $(X^{-1}X_n)_{n\ge 1}$ can be almost surely
approximated by the sequences
$$
(X_1+\cdots+X_m)^{-1}(X_1,\ldots,X_m,0,0,\ldots),\quad m\ge 1,
$$
Theorem \ref{t:weak_conv} shows that $(X^{-1}X_n)_{n\ge 1}$ is NA.
\end{example}
%%%%%%%%%%%%%%%%%%%%%%%%%
\begin{example}\rm (Markov stochastically monotone, up-down processes, Liggett \cite{liggett}, Szekli \cite[ Section 3.8, Theorem A]{szekli}) Let $\mathbf {X}=(X(t), t\ge 0)$ be a time homogeneous Markov Feller process with values in a partially ordered Polish space $\X$ with generator $A$. If $\mathbf X$ is stochastically monotone and up-down (i.e. $Afg\ge fAg+gAf$, for non-decreasing $f,g$) and $X(0)$ is positively associated then $\mathbf X$ is PA, i.e. $(X(t_1),\ldots ,X(t_n))$ is PA as a random element of $\X ^n$, for all $t_1<\ldots <t_n$, $n\in \N$, and the invariant (stationary) distribution of $\mathbf X$ is PA (if it exists). Using our results the PA property can be extended to infinite sequences. In this class of Markov processes many particle systems (attractive) and generalized birth and death processes are included.
\end{example}
%%%%%%%%%%%%%%%%%%%%%%%%%
\begin{example}\label{exrint} \rm (Random integrals)
Let $I$ be a countable index set and let $\mathbb{X}$ be a partially ordered Polish space.
Suppose that $f_y\colon \mathbb{X}\to \R_+$, $y\in I$, is a family
of measurable functions and that $M$ is a random measure
on $\mathbb{X}$. Define a random field $\mathbf{X}=(X_y)_{y\in I}$
by
\begin{align*}
X_y:=\int f_y(x)\,M(dx),\quad y\in I.
\end{align*}
If $M$ is PA, then so is $\mathbf{X}$. For simple functions $f_y$, this is straightforward from the definition of PA and then for arbitrary functions one can use the standard approximation along with our weak convergence result (Theorem \ref{t:weak_conv}).
\end{example}
%%%%%%%%%%%%%%%%%%%%%%%%%%%%%%%%%%%%%%%%%%%%%%
%%%%%%%%%%%%%%%%%%
\subsection{Associated random measures}
%%%%%%%%%%%%%%%%%%%
\begin{example}\label{exPoi}\rm (Poisson process) Let $M$ be a Poisson
process on a Polish space $\mathbb{S}$ with a locally finite intensity
measure $\lambda$. By complete independence, $M$ is NA.
It was stated in \cite{Roy90} (referring to the author's PhD-thesis)
that $M$ is PA. We refer to \cite[Theorem 20.4]{last2017lectures}
for a general version of this result.
In the percolation literature
this is better known as the {\em Harris--FKG inequality} (see \cite{Harris1960} and \cite{Fortuin1971}).
If $\lambda$ is diffuse, then \cite[Theorem 6.14]{last2017lectures}
shows that a Poisson process
is the only simple point process with intensity measure $\lambda$
which is both PA and NA.
\end{example}
%%%%%%%%%%%%%%%%%%%%%%%%%%%%%%%%%%%%%
\begin{example}\label{exmixedPoi} \rm (Mixed Poisson process) Let
$\lambda$ be a locally finite measure on a Polish space $\mathbb{S}$.
Let $X\ge 0$ be a random variable and suppose that $M$ is
a point process on $\mathbb{S}$ such that a.s.\
$\BP(M\in\cdot\mid X)=\Pi_{X\lambda}$, where, for a given
locally finite measure $\nu$ on $\mathbb{S}$, $\Pi_\nu$ denotes
the distribution of a Poisson process with intensity measure $\nu$.
Then $M$ is known as a {\em mixed Poisson process}. We show that
$M$ is PA; see \cite[Example 2.1]{georgiik}. Let $f,g\colon \mathbf{\sS}\to\R$
be measurable bounded and non-decreasing. By
conditioning and Example \ref{exPoi}
\begin{align*}
\BE f(M)g(M)\ge \BE[\BE[f(M)\mid X]\,\BE[g(M)\mid X]]
=\BE[\tilde f(X)\tilde g(X)],
\end{align*}
where $\tilde f(x):=\int f(\mu)\,\Pi_{x\lambda}(d\mu)$, $x\ge 0$,
and the function $\tilde g$ is defined similarly.
A simple thinning argument (see e.g.\ \cite[Corollary 5.9]{last2017lectures})
shows that $\tilde f$ and $\tilde g$ are non-decreasing.
Since a single random variable in a totally ordered space is PA (\cite[Theorem 3.4]{Lindqvist88}),
%**insert reference, e.g. Preston 76**
we obtain that
$$\BE[\tilde f(X)\tilde g(X)]\ge
\BE[\tilde f(X)]\,\BE[\tilde g(X)]= \BE[f(M)]\,\BE[g(M)],$$
as asserted.
\end{example}
%%%%%%%%%%%%%%%%%%%%%%%%%%%%%%
\begin{example}\label{exCox} \rm (Cox processes)
%We generalize the assertion of Example \ref{exmixedPoi}.
Let $\Lambda$ be a random measure on a Polish space $\mathbb{S}$
and let $M$ be a point process on $\mathbb{S}$ such that a.s.\
$\BP(M\in\cdot\mid \Lambda)=\Pi_{\Lambda}$. Then $M$ is known
as a {\em Cox process}. We show that if $\Lambda$ is associated,
then so  is $\Phi$. Assume first that $\Lambda$ is PA.
Let $f,g\colon \mathbf{\sS}\to\R$
be measurable bounded and non-decreasing. Since Poisson processes
are PA we have similarly as in Example \ref{exmixedPoi} that
$\BE[f(M)g(M)]\ge \BE[\tilde f(\Lambda)\tilde g(\Lambda)]$,
where $\tilde f(\nu):=\int f(\mu)\,\Pi_{\nu}(d\mu)$,
$\nu\in\mathbf{M}(\sS)$, and $\tilde g$ is defined
similarly. By the thinning properties of Poisson processes
the (measurable) functions $\tilde f$ and $\tilde g$ are non-decreasing.
Hence $\BE[\tilde f(\Lambda)\tilde g(\Lambda)]\ge \BE[f(M)]\,\BE[g(M)]$ and $M$ is PA.
Assume now that $\Lambda$ is NA and that $f$ and $g$
are measurable with respect to disjoint measurable subsets of $\sS$.
Using in the above calculation the complete independence
of a Poisson process instead of PA, and the fact that for
each measurable set $A$ the restriction $\Phi_A$ is Cox with directing
measure $\Lambda_A$, we obtain that $M$ is NA.
The PA case of this example generalizes Example \ref{exmixedPoi} and,
in fact, Theorem 5.5 in \cite{burtonw85}.
The NA case might be new, at least in this generality. Note
that our strong (functional) definition of association has been crucial
for the above arguments.
\end{example}

%%%%%%%%%%%%%%%%%%%%%%%%%%
\begin{example}\label{expermpoint} \rm (Permanental point processes)
Assume that $\sS$ is a locally compact separable metric space
and let $\mathbf{X}=(X_s)_{s \in \mathbb{S}}$ be a Gaussian random
field. It was shown in \cite{Eisenbaum14} that the
finite-dimensional distributions of $(X^2_s)_{s \in \mathbb{S}}$ are PA
iff they are infinitely divisible. Assume this is the case
and that moreover, $\mathbf{X}$ has continuous sample paths.
Let $\mu$ be a locally finite measure on $\sS$ and define
$\Lambda:=\int\I\{s\in\cdot\}X^2_s\,\mu(ds)$.
It can be shown as in Example \ref{exintrf} below that the random measure
$\Lambda$ is PA. By Example \ref{exCox}, a Cox process $\Phi$ directed by $\Lambda$ is PA.
Such a $\Phi$ is a special case of a (1/2)-{\em permanental process};
see e.g.\ \cite[Chapter 14]{last2017lectures}.
More generally, we may consider $k$ i.i.d.\ infinitely divisible
Gaussian random fields $\mathbf{X}^1,\ldots,\mathbf{X}^k$ as above
and define $\Lambda:=\int\I\{s\in\cdot\}Y_s\,\mu(ds)$,
where $Y_s:=(X^1_s)^2+\cdots+(X^k_s)^2$. By a basic property
of association the field $(Y_s)_{s \in \mathbb{S}}$ is again PA,
so that a Cox process $\Phi$ with directing measure $\Lambda$ is PA
as well. Such a $\Phi$ is $k/2$-permanental; see again \cite[Chapter 14]{last2017lectures}.
\end{example}
%%%%%%%%%%%%%%%%%%%%
\begin{example}\rm (Determinantal point processes, Lyons\  \cite[Theorem 3.7]{lyons2014determinantal})
  Let $\lambda$ be a Radon measure on a locally compact Polish space
  $\X$. Let K be a locally trace-class positive contraction on
  $L_2(\X, \lambda)$.  The determinantal point process defined by K is
  \na \ as a random measure. Well known examples of determinantal
  point processes are descents in random sequences
  (Borodin et al. \cite{borodin}), non-intersecting random
  walks (Johansson, \cite{johansson}), edges in random spanning trees (Burton and
  Pemantle \cite{burton}) and the finite and infinite Ginibre ensemble (\cite{ginibre}, see also Section \ref{s:appendix}).
\end{example}
%%%%%%%%%%%%%%%%%%%%%%%%%%%%%%%%%%%%
\begin{example}\rm (Mixed sampled point processes, Last and Szekli\
  \cite[Theorem 3.3]{last-szeklina}) Suppose that
  $N:= \sum_{i=1}^{\tau}\delta_{X_i}$, where $X_i$ are i.i.d.\ on
  a Polish space $\X$ and $\tau \in \N \cup \{0\}$ is independent of
  $(X_i)_{i \geq 1}$. This is called as
{\em a mixed sampled point process}; see also \cite{last2017lectures}.
If $\tau$ has an ultra log-concave distribution, then
  $N$ is \na \ as a random point process.
  This example can be immediately extended to the case of random
  measures $M := \sum_{i=1}^{\tau}W_i\delta_{X_i},$ for an
  independent iid sequence $(W_i)$ of positive random variables. Such random measures belong to the class of random measures described in the next example.
\end{example}
%%%%%%%%%%%%%%%%%%%%%%%%%%%%%%%%%%%%%
\begin{example}\rm (Independently-weighted point processes) Suppose
   that $N = \{X_i\}_{i \geq 1}$ is a NA point process on $\mathbb{S}$
   and $(W_i)_{i \geq 1}$ is an independent but possibly position
dependent
   marking of $N$ with non-negative marks (see
   \cite[Section 5.2]{last2017lectures} for more details). In other
   words, given $N$, let $(W_i)$ be independent random variables chosen
   as per distribution $K(X_i,.)$, where
   $K(x,dw)$ is the
   probability kernel generating the independent marking. Define the
   random measure $M:= \sum_i W_i \delta_{X_i}$. Clearly we have
a.s.\ that $\BP(M\in\cdot\mid N)=K^*(N,\cdot)$ for a suitably defined
probability kernel $K^*$. Suppose that
$f,g\colon \mathbf{M}(\mathbb{S})\to\R $ are bounded measurable and
non-decreasing.
Assume also that there exists a measurable $A\subset \mathbb{S}$ such
that $f$ is measurable w.r.t.\ $A$ and $g$ is measurable w.r.t.\ $A^c$.
Since $M_A$ and $M_{A^c}$ are conditionally independent given $N$,
we have a.s.\ that $\BE[f(M)g(M)]=\BE[[\BE[f(M)\mid N]\,\BE[g(M)\mid N]]$.
We can define $K^*$ in such a way
that $\int f(\nu)\, K^*(\mu,d\nu)$ and $\int g(\nu)\, K^*(\mu,d\nu)$
are increasing
in $\mu$.
Therefore $M$ is NA.
\end{example}
%%%%%%%%%%%%%%%%%%%%%%%%%%%%%%%%%%%%%
\begin{example}\label{exintrf}\rm (Integral of random fields) Suppose $\X$ is a
  Polish space with a locally finite measure $\mu$ and
  ${\mathbf X} := (X(x))_{x \in \X}$ is a $\Y$-valued continuous
  random field where $\Y$ is a POP space. Assume that
  $(X(x))_{x\in I}$ is NA for any finite $I \subset \X$. Let
  $f : \Y \to \R_+$ be an increasing and continuous function. Then we
  have that the random measure
  $M(A) := \int_{A} f(X(x)) \mu(\md x), A \in \mathcal{S}$, is a NA
  random measure. This can be proved as follows.  Easily we have that
  $(f(X(x)))_{x\in I}$ is NA for any finite $I \subset \X$. Now, we
  approximate $M(A)$ for any $A \in \mathcal{S}_b$ as follows. Let
  $\{x_n\}_{n \geq 1}$ be a countable dense set of $\X$ and
  $B_n^k := B_{x_n}(2^{-k}) \setminus (\cup_{m=1}^{n-1}B_{x_m}(2^{-k})
  )$. Choose $y_n^k \in B_n^k$ for all $n,k$. Define
$$
M_k(A):= \int_{A} \sum_{n=1}^{\infty} \I\{x \in B_n^k\}\I\{y_n^k \in A\}f(X(y_n^k))\mu(\md x).
$$
Observe that $M_k(A)$ is an increasing function of
$\{f(X(y_n^k))\}_{y_n^k \in A}$ and by Theorem \ref{t2.1},
$\{f(X(y_n^k))\}_{y_n^k \in A}$ is a NA random field. Thus for
disjoint bounded sets $A_1,\ldots,A_m$, since $M_k(A_i)$'s are
increasing functions of disjoint collection of $f(X(y_n^k))$'s, we
have that \\ $(M_k(A_1),\ldots,M_k(A_m))$ is NA. By continuity of
$f,\mathbf{X}$ and boundedness of $A_i$'s, we can use the dominated
convergence theorem to show that for all $1 \leq i \leq m$,
$M_k(A_i) \to M(A_i)$ a.s. as $k \to \infty$. Now, by using Theorem 
\ref{t:weak_conv}, we have that $(M(A_1),\ldots,M(A_m))$ is NA for
disjoint bounded sets $A_1,\ldots,A_m$ and hence $M$ is NA by Theorem \ref{measurena}.
\end{example}
%%%%%%%%%%%%
\begin{example}\label{diri}\rm (Dirichlet process) Let $\lambda$ be a measure on
$\sS$ such that $0<\lambda(S)<\infty$. A random measure $M$ on $\sS$
is called a {\em Dirichlet process} \cite{Ferguson73,last2017lectures} with
parameter measure $\lambda$ if
$(M(B_1),\ldots,M(B_n))$ has a Dirichlet distribution
with parameter $(\lambda(B_1),\ldots,\lambda(B_n))$,
whenever $B_1,\ldots,B_n$, $n\ge 1$, form a measurable partition of $\sS$.
By Example \ref{dirichletseq}, a Dirichlet process is NA. Note that
the NA property of Dirichlet sequences is in accordance with
Theorem \ref{measurena}.
\end{example}
%%%%%%%%%%%%%%%%%
\begin{example}\label{infdiv}\rm (Infinitely divisible random measures, Burton and
  Waymire, \cite{burton-way}, Evans, \cite{evans}) Suppose that $M$ is a random
measure
on a Polish space $\mathbb{S}$ which is infinitely divisible. This means
that for any $n\in \N$,  there exist independent identically distributed
random measures
   $M_1, \ldots,M_n$ on $\mathbb{S}$ such that $M$ has the same
distribution as
   $M_1 +\cdots +M_n$. It was shown in \cite{burton-way} and \cite{evans}
that $M$ is PA.
We give here a short proof of this result which does, moreover, not
require $\mathbb{S}$
to be locally compact. By a classical point process result
(see e.g.\ \cite[Theorem 3.20]{kallenberg17}) there exists a Poisson process
$\Phi$ on $\mathbf{M}(\mathbb{S})$ and a measure
$\lambda\in\mathbf{M}(\mathbb{S})$ such that
$M=\lambda+\int \mu\,\Phi(d\mu)$ holds a.s.
Taking measurable bounded and non-decreasing functions
$f,g\colon\mathbf{M}(\mathbb{S})\to\R$, we obtain that
\begin{align*}
\BE[f(M)g(M)]=\BE[\tilde f(\Phi)\tilde g(\Phi)],
\end{align*}
where the function $\tilde f$ (and similarly $\tilde g$) is defined as
follows.
Given a locally finite counting measure $\varphi$ on
$\mathbf{M}(\mathbb{S})$
we set $\tilde f(\varphi):=f\big(\lambda+\int \mu\,\varphi(d\mu)\big)$
whenever the measure $\int \mu\,\varphi(d\mu)$ is locally finite.
Otherwise we set $\tilde f(\varphi):=c$, where $c$ is an upper bound of $f$.
Since $\tilde f$ and $\tilde g$ are non-decreasing we can apply the
PA property of $\Phi$ (see Example \ref{exPoi}) to conclude that
$\BE[\tilde f(\Phi)\tilde g(\Phi)]\ge \BE[\tilde f(\Phi)]\,\BE[\tilde
g(\Phi)]
=\BE[f(M)]\,\BE[g(M)]$, as asserted.
\end{example}
%%%%%%%%%%%%%%%%%%%%%%%%%%
\begin{example}\rm (Poisson cluster random measure)
Suppose that $N=\sum^\tau_{i=1}\delta_{\xi_i}$ is a Poisson process
on a Polish space $\mathbb{S}$. Let
   $(M_i, i\ge 1)$ be an i.i.d.\ sequence of random measures on
$\mathbb{S}$,
independent of $N$. Assume that
\begin{align*}
\iint \min(\mu(B+x),1)\,\BP(M_1\in d\mu)\,\BE[\Phi](dx)<\infty
\end{align*}
for all bounded Borel sets $B\subset \mathbb{S}$. By \cite[Theorem
3.20]{kallenberg17}
the random measure $M$ defined by
$M(B) =\sum^\tau_{i=1} M_i(B + \xi_i)$, $B\in\mathcal{S}$, is infinitely
divisible.
Example \ref{infdiv} shows that $M$ is PA.
\end{example}
\begin{example}\rm (Self-exciting point processes on the real axis,
  Kwiecinski and Szekli \cite[Theorem 4.2]{kwiecinski-szekli96}) Let
  $N$ be a point process on $\R_+$ admitting stochastic intensity with
  respect to its internal filtration. If $N$ is a positively
  self-exciting w.r.t.\ $\prec$, then $N$ is positively associated
  w.r.t.\ $\prec$, whenever $\prec$ denotes one of the three orderings
  of point processes introduced there. In particular renewal processes
  with inter-point distribution which has decreasing failure rate
  (DFR) are \pa \ as random measures.
\end{example}
%%%%%%%%%%%%%%%%%%%%%%%%%%
%%%%%%%%%%%%%%%%%%%

%%%%%%%%%%%%%%%%%%%%
\begin{example}\rm (Area interaction process) Let $\mathbb{S}$ be a compact subset
of $\R^d$ equipped with the Euclidean distance. Let $\beta>0$ and let $\Pi_\beta$
be the distribution of a Poisson process with intensity measure
$\beta\lambda_d$ restricted to $\mathbb{S}$, where $\lambda_d$ denotes Lebesgue
measure on $\R^d$. Fix a number $r>0$ and define
$U(\mu):=\cup_{x\in\mu}B(x,r)$, $\mu\in\mathbf{N}(\sS)$, where
$B(x,r)$ is the Euclidean ball with radius $r$ centred at $x$.
Suppose that $\Phi$ is a point process
on $\mathbb{S}$ whose distribution is absolutely continuous w.r.t.\
$\Pi_\beta$, with density proportional to
$p(\mu)=e^{-\alpha \lambda_d (U(\mu))}$, $\mu\in\mathbf{N}(\sS)$,
where $\alpha>0$ is another parameter.
Example 2.3 in \cite{georgiik} shows the finite dimensional distributions of $\Phi$ are positively associated thus using our Theorem \ref{measurena} we conclude that $\Phi$ is PA. In fact, the latter example covers
a more general class of finite Gibbs processes (of Widom--Rowlinson type) which are PA.
\end{example}
%%%%%%%%%%%%%%%%%%%%%%%%%%%%%%%%%%%%%%
\begin{example}\rm (Exclusion processes)
The symmetric exclusion process on a countable set $S$ is the Markov process $(X_t, t\ge 0)$ on the state space $E=\{0, 1\}^S$
with the formal generator
$$
L f(\eta) =\sum _{x,y:\eta (x)=1, \eta (y)=0} p(x, y)[ f(\eta_{x,y}) - f(\eta)],\  \eta\in E,
$$
where $\eta_{x,y}$ is the configuration obtained from $\eta$ by interchanging the coordinates $\eta (x)$ and $\eta(y)$. Here $p(x, y) = p(y, x)$ are the transition probabilities for a symmetric, irreducible, Markov chain on $S$. For background on this process, see Chapter VIII of \cite{liggett}.
Let
$${\cal H} =\{\alpha : S \to [0, 1], \sum _y p(x, y)\alpha(y) = \alpha(x)\  \forall x \},
$$
and for $\alpha \in {\cal H}$, let $\nu_\alpha$ be the product measure with marginals $\nu_\alpha(\eta : \eta(x) = 1) = \alpha(x)$. Then the limiting distribution as $t\to \infty$  of the process $(X_t)$ exists if the initial distribution of $X_0$ is $\nu_\alpha$; call it  $\mu_\alpha$. It is known from \cite{borceabl} that for $\mu_\alpha$  the finite dimensional distributions are negatively associated and using our Theorem \ref{t2.1} we have that  $\mu_\alpha$ is NA.
\end{example}
%%%%%%%%%%%%%%%%%%%%%%%%%%%%%%%%%%%%%

%%%%%%%%%%%%%%%%%%%%%%%%%%%%%%%%%%%%%%%%%%%
\section{Appendix}\label{s:appendix}
In order to make the list of examples more complete we recall some classical results related  directly to applied probability models.

\begin{enumerate}[a)]
\item (Non-Gaussian infinitely divisible random vectors, Samorodnitsky \cite{samorodnitsky})  Let $\mathbf X$ be an infinitely divisible random vector with  L\'evy measure $\nu$ which is concentrated on the positive $(\R_+)^d$ and the negative $(\R_-)^d$ quadrants of $\R^d$ then $\mathbf X$ is \pa \ . This condition is not necessary in general but it is for some sub-classes of infdiv vectors.
\item (Max infinitely divisible random vectors, Resnick \cite{resnick}) A random vector  $\mathbf X$ is max-infinitely divisible if for every $n \in  \N$ there exist i.i.d. random vectors $\mathbf X_{n1}, \mathbf X_{n2},\ldots , \mathbf X_{nn}$ such that $\mathbf X$ is equal in distribution to $\max (\mathbf X_{n1}, \mathbf X_{n2},\ldots , \mathbf X_{nn})$.
Every max-infinitely-divisible random vector $\mathbf X$  is \pa \ .

\item  (Karlin, Rinott \cite{karlin}). If the distribution of a vector ${\mathbf X}$ has density $f$ such that $f({\bf x}\vee {\bf y})f({\bf x}\wedge {\bf y})\ge  f({\bf x}) f({\bf y}),$  for all ${\bf x},{\bf y}\in \R^n$ it is called multivariate totally positive of order 2 ($MTP_2$). An  $MTP_2$, random vector ${\mathbf X}$ induces an \pa \  set of random variables (FKG inequalities).
The following special cases are classical $MTP_2$ densities: (i) the negative multinomial discrete density; (ii)  ${\mathbf X}$ is normally distributed with mean zero and the covariance matrix $\Sigma$  is $MTP_2$ if and only $-\Sigma ^{-1}$ exhibits nonnegative
off-diagonal
elements (that is $\Sigma ^{-1}$ is so called M-matrix or Leontief matrix); (iii) the density of the eigenvalues of certain Wishart random matrices; (iv) the density of multivariate logistic distribution; (v) the density of the multivariate gamma distribution; (vi) the density of the multivariate Cauchy distribution.

\item (Virtual waiting time process, Kwiecinski and Szekli \cite[Proposition 5.1]{kwiecinski-szekli96}) Suppose that a marked point process $N$ feeding a single-server queue is positively associated as random measure. Then the processes of the virtual waiting time and  of the number of customers in the system are PA as random fields.

\item (M-infinitely divisible random sets, Karlowska-Pik and Schreiber, \cite[Theorem 2.1]{karlowska}) If M-infinitely-divisible convex compact random set $\mathbf X$ has no Gaussian summand and its L\'évy measure concentrates on the family of sets containing the origin, then $\mathbf X$ is \pa \ as a random element of the space of closed subsets of $\R^d$ equipped with the Fell topology. Similarly, every union-infinitely-divisible random closed set is \pa\ .

\item (Sojourn times on quasi overtake–free paths in queueing networks, Daduna and Szekli \cite[Theorem 4.2 and 6.4]{daduna}).
The vector of a test customer's successive sojourn times on a quasi overtake-free path in a closed Gordon-Newell queueing network is negatively associated. In particular, the vector of a test customer's successive sojourn times in a cycle is NA.

\item (Queueing networks, Szekli \ \cite[Section 3.8, Theorem E]{szekli}) Let $(X(t), t \ge 0$) denote the joint queue length process of an irreducible Gordon-Newell network with Markovian routing and queue-length dependent non-decreasing service rates, which acts in equilibrium. Then for each $t\ge  0$, $X(t)$ is NA.

\item (Eigenvalues of random matrices, Ginibre \cite{ginibre}) Let $M$ be a random matrix obtained by drawing every entry independently from the complex normal distribution. This is the complex Ginibre ensemble. The eigenvalues of $M$, which form a finite subset of the complex plane define a NA point process (which is  determinantal). If a Hermitian matrix is generated in the corresponding way, drawing each diagonal entry from the normal distribution and each pair of off-diagonal entries from the complex normal distribution, then we obtain the Gaussian unitary ensemble, and the eigenvalues are now a NA (determinantal)  point process on the real line.
\item (Van den Berg and Kesten (BK) inequality) Let $E= \{0, 1\}^n$, and  $[n]:=\{1,..., n\}$. For $\eta \in E$  and $I \subset [n]$, let $\eta_I$ denote the ‘tuple’ $ (\eta_i , i\in I)$. By $[\eta]_I := \{\alpha\in E: \alpha_I = \eta_I\}$ we denote the set of all elements of $E$ that agree with $\eta$ on $I$. $ A\square B$ defines the event that A and B occur disjointly, that is
$$
A\square B =\{\eta\in E:\exists \ {\rm disjoint}\  K, L \subset [n],  [\eta]_K \subset A ,  [\eta]_L \subset B\}.
$$
An event $A \subseteq E$ is
said to be increasing if $\eta' \in A$ whenever $\eta\in A$ and $\eta'\ge \eta$ coordinatewise. A probability measure $P$ on $E$ is BK if for all  increasing $A, B$, $P(A\square B) \le P(A)P(B)$. It is known that if $P$ is BK then it is NA but NA does not imply BK, see \cite{bk}.
\item (Distributions on vertices of polytopes in $\R^n$, Peres et al. \cite{peres}) For a Gaussian random walk in a polytope that starts at a point inside and continues until it gets absorbed at a vertex the probability distribution induced on the vertices by this random walk is NA for matroid polytopes. Such distributions are highly sought after in randomized algorithms as they imply concentration properties.
\item (Random-cluster model, Grimmett \cite{Grimmett2006}) The random cluster measure $\phi_{p,q}$ is PA for all $p \in [0,1],  q \in [1,\infty)$ with free or wired boundary conditions. For any other boundary condition, the limit random cluster measures and extreme (tail trivial) DLR random cluster measures are PA for all $p \in [0,1],  q \in [1,\infty)$ (see \cite[Theorems 4.17 and 4.37]{Grimmett2006}).  It is one of the important conjectures in statistical physics that $\phi_{p,q}$ satisfies some form of negative dependence for $p \in [0,1],  q \in (0,1)$. From our Theorem \ref{t:weak_conv}, this conjecture boils down to showing a suitable negative dependence property for the finite-volume case. However, this is shown in certain special cases of the $q \downarrow 0$ limit (see \cite[Section 3.9]{Grimmett2006}).
\item (Conditional distributions, Hu and Hu \cite{hu})
Let ${\mathbf X}=(X_1,X_2,\ldots, X_n)$ be a random vector of $n$ iid rv’s with a continuous distribution. Then  $[{\mathbf X} | X_{(k_1)} = s_1,X_{(k_2)} = s_2,\ldots, X_{(k_r)} = s_r]$ is NA for  $1\le k_1<k_2<\ldots <k_r\le n$  and $s_1<s_2<\ldots<s_r$, where $X_{(1)}\le \ldots \le X_{(n)}$ are the order statistics of ${\mathbf X}$.
If ${\mathbf X}$ is  a random vector of $n$  iid rv’s with PF2 densities or mass functions then $ [{\mathbf X} | \sum _{i=1}^n X_i \in (a, b)]$ is NA, where $a<b$. Some special cases of conditional distributions are given in the next example.
\item (Joeg-Dev and Proschan, \cite{Joag-Dev1983})
Random vectors ${\mathbf X}$ with the permutation, multinomial, multivariate hypergeometric or Dirichlet  distributions are \na . For Dirichlet, see also Example \ref{diri}.

\section*{Acknowledgments}
DY's work was supported by INSPIRE Faculty Award from DST and CPDA
grant from the Indian Statistical Institute.  DY also wishes to thank
Bartek Blaszczyszyn and Subhrosekhar Ghosh for some discussions on
negatively associated point processes. RS's work was supported by
National Science Centre, Poland, grant NCN no
2015/19/B/ST1/01152. This work was in part supported by the German
Research Foundation through Grant No.\ LA965/9-2 awarded as part of
the DFG-Forschungs\-gruppe FOR 1548 ``Geometry and Physics of Spatial Random Systems''. The authors also wish to thank an anonymous referee for pointing out errors in the proofs of Proposition 3.3 (now Lemma 3.5) and Theorem 3.6 of the earlier version. 
\end{enumerate}

\end{document}